\documentclass[a4paper,12pt]{amsart}
\usepackage{amssymb,amsfonts,amsmath}
\usepackage{layout}
\usepackage[latin1]{inputenc}
\usepackage[T1]{fontenc}
\usepackage{amsthm}
\def\m{\mathcal}

\def\c2{\mathbb{C}^2}
\def\R{\mathbb{R}}

\def\1{\bold{1}}

\def\a{\alpha}

\def\f{\varphi}

\newcommand \w {\omega}

\newcommand \sm {\setminus}

\newtheorem{lem}{Lemma}[section]

\theoremstyle{definition}

\theoremstyle{plain}
\newtheorem{def/not}[lem]{Definition/Notations}
\numberwithin{equation}{section}

\newtheorem{cor}[lem]{Corollary}

\theoremstyle{definition}

\theoremstyle{plain}

\newenvironment{proof3.1}
{\noindent {\it{Proof of theorem 3.1}}}{$\Box$ \linebreak[4]}

\begin{document}

\title[ Complex Monge-Amp\`ere  Equations]
{Weak Solutions to Complex Monge-Amp\`ere  Equations on 
Compact K\"ahler Manifold }
\author{ Slimane BENELKOURCHI }
\address{ Universit\'e de Montr\'eal,
Pavillon 3744, rue Jean-Brillant,
Montr\'eal QC  H3C 3J7
}
\email{slimane.benelkourchi@umontreal.ca}
\maketitle
\begin{abstract}
We show a general existence theorem to the complex
Monge-Amp\`ere type equation on compact K\"ahler manifolds.
\end{abstract}
\section{Introduction}
Let $(X, \w)$ be a compact K\"ahler manifold of dimension $n.$ 
Throughout this note, $\theta$ denotes a smooth closed form of bidegree
$(1, 1)$ which is nonnegative and big, i.e. such that 
$\int _X \theta ^n >0.$ 
Recall that a $\theta -plurisubharmonic $ ($\theta-$psh for short)
 function is an upper semi-continuous function $\f$ such that
  $\theta +dd^c \f $ is
 nonnegative in the sense of currents.
 The set of all $\theta -psh$ functions $\f $ on $X$
  will be denoted by $PSH(X, \theta)$ and
endowed with the weak topology, which coincides with
 the $L^p (X)-$topology.
We shall consider the existence and uniqueness 
of the weak solution to the following complex Monge-Amp\`ere equations  
\begin{equation} \label{cma}
(\theta +dd ^c \f)^n = F( \f , \cdot) d \mu 
\end{equation}
 where $\f$ is a $\theta $-psh function, $F(t , x) \ge 0$ is a measurable function on $ \R \times X$
 and $ \mu $ is a positive measure.
It is well known that we can not always  make sense to the left hand side of (\ref{cma}) as nonnegative measure.
But according to \cite{BT 2}
  (see also \cite{BBGZ}, \cite{BEGZ}, \cite{GZ7}) , we can define the non
pluripolar product $(\theta + dd^c u)^n$ as the limit of 
$\mathbf{1}_{ (u>-j) }(\theta + dd^c (\max(u, -j))^n.$
 It was shown in \cite{BEGZ} that its trivial extension is nonnegative closed current and 
$$
  \int_X (\theta + dd^c u)^n \le \int _
 X \theta ^n.
 $$ 
 Denote by $\m E(X, \theta ) $ the set of all $\theta -$psh with full
 non-pluripolar Monge-Amp\`ere measure i.e. $\theta -$psh functions for which  the last inequality becomes equality.
 
For $F$ smooth and $ \mu = dV$ is a smooth positive volume form, the equation has been studied 
extensively by various authors, see for example \cite{A1}, \cite{A2}, \cite{BEGZ}, \cite{K05}, \cite{K00}, \cite{Lu},  \cite{Yau}
... and reference
therein. Recently, Ko\l odziej
 treated the case $F$  bounded by a function independent of
the first variable and  $\mu = \w^n,$ where $\w$ is a K\"ahler form on $X.$
In this paper we consider a more general case.
   Our main purpose  is to prove
the following theorem.
\\
\noindent {\it {\bf Main  Theorem.}
 Assume that $F : \R \times X \to [0, +\infty )$ 
 is a measurable function such that:\\
1) For all $x\in X $ the function $t \mapsto F(t, x)$ 
is continuous and  nondecreasing;\\
2) $ F(t,\cdot) \in L^1 (X, d\mu)$  for all $t\in \R ;$\\
3) $$ 
\lim_{t\to -\infty} \int _X F(t,x) d \mu \le  
\int _X \theta^n  < \lim _{t\to +\infty} \int_X F(t, x) d\mu.
$$
Then there exists a unique (up to additive constant) $\theta -psh $ function $\f \in \m E(X, \theta)$
solution to 
$$
(\theta +dd ^c  \f)^n = F( \f , \cdot) d \mu .
$$
}
\section{Proof}
\begin{lem}
Let $\mu $ be a positive measure on $X$ vanishing on all
pluripolar subsets of $X$  and
 $ u_j \in \m E(X, \theta ) $ 
such that  $u_j \ge u_0 $ for some $u_0 \in \m E(X, \theta)\cap L^1(d \mu).$
 
If $u_j \to u $ in $L^1(X) ,$ then 
$$
\lim_{j\to +\infty} \int _X u_j d \mu = \int _X u d \mu .
$$
\end{lem} 
\begin{proof}
Since $u_0 \in   L^1(d \mu)$ and the measure $\mu$ puts no mass on pluripolar subsets of $X$ then 
$$
\int _\a ^{+\infty} \int _{(u_j <-t ) } d \mu dt
\le \int _\a ^{+\infty} \int _{(u_0 <-t ) } 
d \mu dt
  \to 0 \quad \mbox{ as }
 \a \to +\infty. 
$$
Hence, by Dunford-petits theorem (see for example \cite{Ed}  p. 274), we have the sequence 
$(u_j)$ is weakly relatively compact in $L^1(d \mu).$ Let $\hat{u} \in L^1(d \mu)$ be 
a cluster point of $(u_j).$ After extracting a subsequence, we may assume that 
$(u_j)$ converges to $\hat{u}$ weakly in $L^1(d \mu).$
On the other hand, we have $u_j \to u $ in $L^1(X) .$ So, choosing a subsequence if necessary,
we can assume that $u_j \to u $ point-wise on $X\sm A,$  
where $A = \{ \limsup _{j\to \infty} u_j <u\}.$ 
But $ A$ is negligible, hence, by \cite{BT 1} $A$ is pluripolar subset of $X,$
thus $\mu (A)=0.$ It follows from the  Lebesgue's dominated convergence theorem  that
 $u_j \to u $ weakly in $L^1(d \mu) .$ Therefore 
 $\hat{u} = u \ \mu -a.e.$ Hence $u$ is the unique cluster point of $(u_j),$
 which means that  $(u_j)$ converges to ${u}$ weakly in $L^1(d \mu)$ and the proof is complete.
 \end{proof}
The following corollary is the global 
version of Corollary 1.4 in \cite{CK6}. 
\begin{cor} \label{Ceg-kol}
Let $\mu $ be a nonnegative measure which puts no mass on pluripolar sets of $X.$
Then for any sequence $u_j \in \m E (X, \theta)$ converging weakly, one can extract a subsequence
which converges pointwise $d\mu -$ almost everywhere.
\end{cor}
\begin{proof}
[Proof of Main Theorem.] The set 
 of $\f \in  PSH(X, \theta)$
 normalized by
$\sup_X \f = 0 $ is compact (cf. \cite{GZ5}, \cite{GZ7}).
 Then there exists a
positive constant  $C_0>0$
such that 
$$
\int _X - u \theta^n  \le C_0, \quad \forall u \in PSH(X, \theta)
; \quad \sup _X u =0.
$$
 Consider the set 
$$
\m H := \left\lbrace \f \in PSH (X, \theta);  \f \le 0 \ \text{and\ }\ 
\int _X -\f \theta ^n \le C_0 \right\rbrace
$$
It is obvious that $\m H$ is compact convex subset of $L^1(X).$

From the conditions of the main theorem, 
there exists a real number $c_0$ such that 
$$
\int_X  F( c_0 ,\cdot) d \mu= \int_X\theta^n. 
$$ 
Fix a function $\f \in \m H.$ 
 Then there exists a 
real  number $c_\f \ge c_0$ such that 
$$
\int _X F(\f + c_\f, \cdot) d\mu  = \int _X \theta ^n .
$$
Since $F(\f + c_\f, \cdot) \in  L^1 (X, d\mu )  $ and $\mu$ vanishes on pluripolar sets, it follows by \cite{BEGZ} \cite{BGZ} that 
 there exists a  function
 $\tilde \f \in \m E(X, \theta)$ such that $\sup_X \tilde \f  = 0$ and 
$$
(\theta + dd ^c \tilde \f )^ n = F(\f + c_\f ,\cdot) d \mu.
$$
The function $\tilde \f$ does not depend in the constant $c_\f.$
Indeed, assume that there exist two constant $c_\f$ and $c^\prime _\f$ such that 
$$
\int_X F(\f + c_\f ,\cdot) d \mu=\int_X  F(\f + c^\prime_\f ,\cdot) d \mu= \int_X
\theta^ n  .
$$
If $c_\f \le c^\prime _\f$ then $F(\f + c_\f ,\cdot) d \mu \le   F(\f + c^\prime_\f ,\cdot) d\mu $. 
Thence
$$ F(\f + c_\f ,\cdot) d \mu=  F(\f + c^\prime_\f ,\cdot) d \mu. $$
By the uniqueness result in \cite{BEGZ}  and \cite{Di},
 we get that $\tilde \f$ is unique and therefore independent of the constant $c_\f .$
 
From the definition of $\m H$ we have $\tilde \f \in \m H.$ 
Consider the map $\Phi : \m H \to \m H$ defined 
by $\f \mapsto \tilde \f .$ In fact, the range of $\Phi $ is equal to $ \m H \cap \m E(X , \theta).$ 

We claim that $\Phi $ is continuous.
Indeed, let $\f_j\in \m H$ be a converging sequence
 with limit 
$\f \in \m H $ in $L^1 (X)$-topology.
Let $\psi $ be any cluster point of the sequence $\tilde \f_j:
= \phi (\f_j ).$ 
We may assume, up to extracting,
that $\tilde \f _j $ converges towards $\psi $ in $L^1(X).$
Since the measure $\mu $ vanishes on pluripolar subsets then by
Corollary \ref{Ceg-kol} above
we can extract a subsequence, which 
still denoted by $\f_j $, so that $\f_j \to \f $  $\mu$-a.e.
 We claim that  the sequence $c_{\f_j}$ is bounded. Indeed, by construction we have $c_{\f_j}\ge c_0.$
 Now if $c_{\f_j} \to +\infty$ then 
$$
\int_X \theta^n = \liminf _{j\to \infty}\int_X  F(  \f_j + c_{\f_j} ,\cdot) d \mu > \int_X \theta^n  , 
$$
which is impossible.
 
 So by passing to a subsequence,  we may assume that 
 $c_{\f_j} \to c.$ Therefore
$F(\f_j + c_{\f_j}, \cdot ) \to F(\f + c , \cdot)$ in $L^1(d\mu).$
Since $\tilde \f_j \to \psi $ in $L^1 (X),$
 then $\psi = ( \limsup _{j\to +\infty } \tilde \f _j )^* $ and 
 therefore by Hartogs' lemma $\sup_X \psi = 0.$ 
 It follows from Proposition 3.2 in \cite{BEGZ} that $\psi \in \m E(X, \theta).$
 Let denote 
 $\psi_j := (\sup_{k\ge j} \tilde \f_k)^* = 
 (\lim_{l\to +\infty} \max_{l\ge k\ge j} \tilde \f_k )^*
 .$
Since the set $ (\sup_{k\ge j} \tilde \f_k < (\sup_{k\ge j} \tilde \f_k)^* )$ is pluripolar then
 by the continuity of the complex Monge-Amp\`ere operator
  along  monotonic
 sequences we have
 \begin{eqnarray*}
 (\theta + dd ^c \psi)^n &= &\lim_{j\to +\infty}
 (\theta + dd ^c \psi_j)^n\\
 &=&\lim_{j\to +\infty} \lim_{l\to +\infty}
 (\theta + dd ^c \max_{l\ge k\ge j} \tilde \f_k)^n \\
 &\ge & \lim_{j\to +\infty} \lim_{l\to +\infty} 
 \min_{l \ge k\ge j} F(\f_k + c_{\f_k}, \cdot)d\mu \\
 &=&\liminf_{j\to +\infty} F(\f_j + c_{\f_j}, \cdot)d\mu\\
 &=&F(\f + c_\f , \cdot) d \mu.
  \end{eqnarray*} 
Thence $(\theta + dd ^c \psi)^n = (\theta + dd ^c \tilde\f)^n.$
 By  uniqueness (shown in \cite{BEGZ})
we get  $\tilde \f = \psi$
and therefore $\Phi$ is continuous. Now, Shauder's fixed
 point theorem implies that there exists a function
$u\in \m H$ such that  $\Phi (u)= u.$ Since $\Phi (\m H) \subset \m E(X, \theta) $ we have $u \in \m E(X, \theta)$ and 
$$
(\theta + dd ^c u )^ n = F(u + c_u ,\cdot) d \mu.
$$
The function $\f := u+c_u$ is the required solution.

Uniqueness follows in a classical way from the comparison principle 
 \cite{BT 1} and  its generalization \cite{BEGZ}.
Indeed assume that there exist tow solutions
 $\f _1$ and $\f_2$ in $\m E(X, \theta )$ such that 
 \begin{equation} \label{eq1}
 (\theta + dd^c \f_i)^n = F(\f_i,.),\quad i=1, 2.
 \end{equation}
 Then, since $F$ is non-decreasing with respect to the first variable, we have 
 \begin{equation} \label{eq2}
  F(\f_1 , .)d\mu \le  F(\f_2 , .)d\mu \ \text{on}\ {(\f_1<\f_2)}. 
 \end{equation}
 On the other hand, by the comparison principle we have
 \begin{equation} \label{eq3}
  \int_{(\f_1<\f_2)} (\theta + dd^c \f_2)^n
 \le \int_{(\f_1<\f_2)} (\theta + dd^c \f_1)^n .
 \end{equation}
 Combining (\ref{eq1}), (\ref{eq2}) and ( \ref{eq3}), we get
$$
\int_{(\f_1<\f_2)} F(\f_1 , .)d\mu \le \int_{(\f_1<\f_2)} F(\f_2 , .)d\mu 
\le \int_{(\f_1<\f_2)} F(\f_1 , .)d\mu .
$$ 
Hence
  $$
 F(\f_1 , .) = F(\f_2 , .) \mu-\text{almost everywhere} \quad \text{on}\ (\f_1 < \f_2) .
  $$
 In the same way, we get the equality on $(\f_1>\f_2)$ and then on $X.$ Hence
 $$(\theta + dd^c \f_1 )^n =(\theta + dd^c \f_2)^n.$$
 Once more, the uniqueness result of 
   \cite{BEGZ} and \cite{Di} implies that
   $\f_1 - \f_2 = Cst$.
\end{proof}
\section*{Acknowledgements}
This note was written during the
author's visit to Institut de Math\'ematiques de Toulouse . 
  He wishes to thank Vincent Guedj and Ahmed Zeriahi  for fruitful discussions
and the warm hospitality. He also wishes to thank the referee 
for his careful reading and for his remarks which helped to
improve the exposition.

\end{document}